\documentclass[journal]{IEEEtran}
\usepackage{amsfonts,amssymb, amsthm,amsmath}
\usepackage{dsfont}
\usepackage{pst-all}
\usepackage{pstricks}
\usepackage{graphicx}
\usepackage{color}

\ifCLASSINFOpdf
\else
\fi
\newtheorem{thm}{Theorem}
\newtheorem{cor}{Corollary}
\newtheorem{lem}{Lemma}
\newtheorem{exa}{Example}
\newtheorem{rem}{Remark}

\newtheorem{defn}{Definition}
\newtheorem{prop}{Proposition}
\usepackage[english]{babel}

\newcommand{\R}{\mathbb R}
\newcommand{\C}{\mathbb C}
\newcommand{\T}{\mathbb T}
\newcommand{\N}{\mathcal N}
\newcommand{\G}{\mathcal G}
\newcommand{\V}{\mathcal V}
\newcommand{\E}{\mathcal E}

\begin{document}
\title{Complex Laplacians and Applications in Multi-Agent Systems \thanks{This work is supported by the National Natural
Science Foundation of China under grant~11301114 and Hong Kong Research Grants Council under grant~618511.}}
\author{Jiu-Gang~Dong, and~Li~Qiu,~\IEEEmembership{Fellow,~IEEE}%
\thanks{J.-G. Dong is with the Department of Mathematics, Harbin Institute of
Technology, Harbin 150001, China and is also with the Department of Electronic and Computer Engineering,
The Hong Kong University of Science and Technology, Clear Water Bay,
Kowloon, Hong Kong, China (e-mail: jgdong@hit.edu.cn).}%
\thanks{L. Qiu is with the Department of Electronic and Computer Engineering,
The Hong Kong University of Science and Technology, Clear Water Bay,
Kowloon, Hong Kong, China (e-mail: eeqiu@ust.hk).}
}

\maketitle

\begin{abstract}
Complex-valued Laplacians have been shown to be powerful tools in the study of distributed coordination of multi-agent systems in the plane including formation shape control problems and set surrounding control problems. In this paper, we first provide some characterizations of complex Laplacians. As an application, we then establish some necessary and sufficient conditions to ensure that the agents interacting on complex-weighted networks converge to consensus in some sense. These general consensus results are used to discuss some multi-agent coordination problems in the plane.
\end{abstract}


\section{Introduction}
In the past decade there has been increasing interest in studying the distributed coordination and control of multi-agent systems, which appear in diverse situations including consensus problems, flocking and formation control~\cite{JLM03,OSM04,OSFM07,OT11}. As a natural tool, Laplacian matrices of a weighted graph (modeling the interaction among agents) are extensively used in the study of the distributed coordination problems of multi-agent systems. Most results are based on real Laplacians, see, e.g., the agreement~\cite{BS03,OSM04,OSFM07,RB05}, generalized consensus~\cite{CLHY11,Morbidi13} and bipartite consensus on signed graphs~\cite{Altafini13,MSJCH14}. Very recently, complex Laplacians have been applied to multi-agent systems~\cite{LDYYG13,LWHF14,LH14}. In particular, formation shape control problems in the plane with complex Laplacians were discussed in~\cite{LDYYG13,LWHF14}, while based on complex Laplacians, new methods were developed in~\cite{LH14} for the distributed set surrounding design, which contains consensus on complex-valued networks as a special case.

It has been shown that complex Laplacians are powerful tools for multi-agent systems and can significantly simplify the analysis once the state space is a plane.
From this point, it is worth investigating complex Laplacians independently. The main goal of this paper is to study the properties of complex Laplacians. More precisely, for a complex-weighted graph, we provide a necessary and sufficient condition ensuring that the complex Laplacian has a simple eigenvalue at zero with a specified eigenvector. The condition is in terms of connectivity of graphs and features of weights. It is shown that the notion of {\em structural balance} for complex-weighted graphs plays a critical role for establishing the condition. To demonstrate the importance of the obtained condition, we apply the condition to consensus problems on complex-weighted graphs. A general notion of consensus, called {\em complex consensus}, is introduced, which means that all limiting values of the agents have
the same modulus. Some necessary and sufficient conditions for complex consensus are obtained. These complex consensus results extend and complement some existing ones including the standard consensus results~\cite{RB05} and bipartite consensus results~\cite{Altafini13}.

This paper makes the following contributions. 1) We extend the known results on complex Laplacains (see~\cite{Reff12}) to a general setting. 2) We establish general consensus results, which are shown to be useful in the study of distributed coordination of multi-agent systems in the plane such as circular formation and set surrounding control. In particular, our results supplement the bipartite consensus results in~\cite{Altafini13}.

The remainder of this paper is organized as follows. Section~\ref{section: complex Laplacians} discusses the properties of the complex Laplacian. Some multi-agent coordination control problems, based on the complex Laplacian, are investigated in Section~\ref{section: application}. Section~\ref{section: examples} presents some examples to illustrate our results.
This paper is concluded in Section~\ref{section: conclusion}.

The notation used in the paper is quite standard.
Let $\R$ be the field of real numbers and $\C$ the field of
complex numbers. For a complex matrix $A\in\C^{n\times n}$, $A^*$
denotes the conjugate transpose of $A$. We use $\bar{z}$ to denote
the complex conjugate of a complex number $z$. The modulus of $z$ is denoted by $|z|$. Let $\mathbf{1}\in\R^n$ be the $n$-dimensional column vector of ones. For $x=[x_1,\ldots,x_n]^T\in\C^n$, let $\|x\|_1$ be its $1$-norm, i.e., $\|x\|_1=\sum_{i=1}^n|x_i|$. Denote by $\T$ the unit circle, i.e., $\T=\{z\in\C:\ |z|=1\}$. It is easy to see that $\T$
is an abelian group under multiplication. For $\zeta=[\zeta_1,\ldots,\zeta_n]^T\in\T^n$, let $D_\zeta:=\mathrm{diag}(\zeta)$ denote the diagonal matrix with $i$th diagonal entry $\zeta_i$. Finally, we have ${\rm j}=\sqrt{-1}$.

\section{Complex-weighted graphs}\label{section: complex Laplacians}

In this section we present some interesting results on complex-weighted graphs. We believe that these results themselves are also interesting from the graph theory point of view. Before proceeding, we introduce some basic concepts of complex-weighted graphs.

\subsection{Preliminaries}
The digraph associated with a complex matrix $A=[a_{ij}]_{n\times n}$ is denoted by $\G(A)=(\V,\E)$, where $\V=\{1,\ldots,n\}$ is the vertex set and
$\E\subset\V\times \V$ is the edge set. An edge $(j,i)\in\E$, i.e., there exists an edge from $j$ to $i$ if and only if $a_{ij}\neq0$. The matrix $A$ is usually called the adjacency matrix of the digraph $\G(A)$.
Moreover, we assume that $a_{ii}=0$, for $i=1,\ldots,n$, i.e., $\G(A)$ has no self-loop. For easy reference, we say $\G(A)$ is complex, real and nonnegative if $A$ is complex, real and (real) nonnegative, respectively.
Let $\N_i$ be the neighbor set of
agent $i$, defined as $\N_i=\{j:\ a_{ij}\neq0\}$. A directed path in
$\G(A)$ from $i_1$ to $i_k$ is a sequence of distinct vertices
$i_1,\ldots,i_k$ such that $(i_l,i_{l+1})\in\E$ for
$l=1,\ldots,k-1$. A cycle is a path such that
the origin and terminus are the same.
The {\em weight} of a cycle is defined as the product of weights on
all its edges. A cycle is said to be {\em positive} if it has a positive weight.  The following definitions are used throughout this paper.
\begin{itemize}
\item[$\cdot$]
A digraph is said to be {\em (structurally) balanced} if all cycles are positive.
\item[$\cdot$]
A digraph has a directed spanning tree if there exists at
least one vertex (called a root) which has a directed path to all other vertices.
\item[$\cdot$]
A digraph is strongly connected if for any two distinct vertices $i$ and
$j$, there exists a directed path from $i$ to $j$.
\end{itemize}
For a strongly connected graph, it is clear that all vertices can serve as roots. We can see that being strongly connected is stronger than having a directed spanning tree and they are equivalent when $A$ is Hermitian.

For a complex digraph $\G(A)$, the complex Laplacian matrix $L=[l_{ij}]_{n\times n}$ of $\G(A)$ is defined by $L=D-A$ where $D=\mathrm{diag}(d_1,\ldots,d_n)$ is the modulus degree matrix of $\G(A)$ with $d_i=\sum_{j\in\N_i}|a_{ij}|$.
This definition appears in the literature on gain graphs (see, e.g., \cite{Reff12}), which can be thought as a generalization of standard Laplacian matrix of nonnegative graphs. We need the following definition on {\em switching equivalence} \cite{Reff12, Z89}.
\begin{defn}\label{defn: switching equivalent}
{\rm
Two graphs $\G(A_1)$ and $\G(A_2)$ are said to be {\em switching equivalent}, written as $\G(A_1)\sim\G(A_2)$, if there exists a vector $\zeta=[\zeta_1\ldots,\zeta_n]^T\in\T^n$ such that $A_2=D_\zeta^{-1} A_1D_\zeta$.}
\end{defn}
It is not difficult to see that the switching equivalence is an equivalence relation. We can see that switching equivalence preserves connectivity and balancedness.
We next investigate the
properties of eigenvalues of complex Laplacian $L$.

\subsection{Properties of the complex Laplacian}

For brevity, we say $A$ is {\em essentially nonnegative} if $\G(A)$ is switching equivalent to a graph with a nonnegative adjacency matrix. By definition, it is easy to see that $A$ is essentially nonnegative if and only if there exists a diagonal matrix $D_\zeta$ such that $D_\zeta^{-1} AD_\zeta$ is nonnegative. By the Ger\v{s}gorin disk theorem \cite[Theorem 6.1.1]{HJ87}, we see that all the eigenvalues of the Laplacian matrix $L$ of $A$ have nonnegative real parts and zero is the only possible eigenvalue with zero real part. We next further discuss the properties of eigenvalues of $L$ in terms of $\G(A)$.

\begin{lem}\label{lem: 1}
Zero is an eigenvalue of $L$ with an eigenvector $\zeta\in\T^n$ if and only if $A$ is essentially nonnegative.
\end{lem}

\begin{proof}
(Sufficiency) Assume that $A$ is essentially nonnegative. That is, there exists a diagonal matrix $D_\zeta$ such that $A_1=D_\zeta^{-1}AD_\zeta$ is nonnegative.
Let $L_1$ be the Laplacian matrix of the nonnegative matrix $A_1$ and thus $L_1\mathbf{1}=0$. A simple observation shows that these two Laplacian matrices are similar, i.e., $L_1=D_\zeta^{-1}LD_\zeta$. Therefore, $L\zeta=0$.

(Necessity) Let $L\zeta=0$ with $\zeta\in\T^n$. Then we have $LD_\zeta\mathbf{1}=0$ and so $D_\zeta^{-1}LD_\zeta\mathbf{1}=0$. Expanding the equation $D_\zeta^{-1}LD_\zeta\mathbf{1}=0$ in component form, we can verify that $D_\zeta^{-1}LD_\zeta\in\R^{n\times n}$ has nonpositive off-diagonal entries. This implies that $A_1=D_\zeta^{-1}AD_\zeta$ is nonnegative and thus $A$ is essentially nonnegative.
\end{proof}

If we take the connectedness into account, then we can derive a stronger result.

\begin{prop}\label{prop: 2}
Zero is a simple eigenvalue of $L$ with an eigenvector $\xi\in\T^n$ if and only if $A$ is essentially nonnegative and $\G(A)$ has a spanning tree.
\end{prop}

\begin{proof}
The proof follows from a sequence of equivalences:
$$
(1)\Leftrightarrow(2)\Leftrightarrow(3)\Leftrightarrow(4).
$$
Conditions (1)-(4) are given in the following.
\begin{itemize}
\item[$(1)$]
$A$ is essentially nonnegative and $\G(A)$ has a spanning tree.
\item[$(2)$]
There exists a diagonal matrix $D_\zeta$ such that $A_1=D_\zeta^{-1}AD_\zeta$ is nonnegative and $\G(A_1)$ has a spanning tree.
\item[$(3)$]
There exists a diagonal matrix $D_\zeta$ such that $L_1=D_\zeta^{-1}LD_\zeta$ has a simple zero eigenvalue with an eigenvector being $\mathbf{1}$.
\item[$(4)$]
$L$ has a simple zero eigenvalue with an eigenvector $\zeta\in\T^n$.
\end{itemize}
Here, the second one is from \cite[Lemma 3.1]{Ren07} and the last one follows from the similarity.
\end{proof}

Here a key issue is how to verify the essential nonnegativity of $A$. Thanks to the concept of balancedness of digraphs, we can derive a necessary and sufficient condition for $A$ to be essentially nonnegative.
To this end, for a complex matrix $A$, we denote by $A_H=(A+A^*)/2$ the Hermitian part of $A$. Clearly, we have $A=A_H$ when $A$ is Hermitian.

\begin{prop}\label{prop: 1}
The complex matrix $A=[a_{ij}]_{n\times n}$ is essentially nonnegative if and only if $\G(A_H)$ is balanced and $a_{ij}a_{ji}\geq0$ for all $1\leq i,j\leq n$.
\end{prop}

\begin{proof}
Since $A_H$ is Hermitian, it follows from~\cite{Z89} that $\G(A_H)$ is balanced if and only if $A_H$ is essentially nonnegative. Therefore, to complete the proof, we next show that $A$ is essentially nonnegative if and only if $A_H$ is essentially nonnegative and $a_{ij}a_{ji}\geq0$ for all $1\leq i,j\leq n$.

{\em Sufficiency:} By the condition that $a_{ij}a_{ji}\geq0$, we have that $|a_{ij}a_{ji}|=\bar{a}_{ij} \bar{a}_{ji}$. Multiplying both sides by $a_{ij}$, we obtain that $|a_{ji}|a_{ij}=|a_{ij}|\bar{a}_{ji}$. Consequently, for a diagonal matrix $D_\zeta$ with $\zeta=[\zeta_1,\ldots,\zeta_n]^T\in\T^n$, we have for $a_{ij}\neq0$
\begin{equation}\label{eq: relation}
\zeta_i^{-1}\frac{a_{ij}+\bar{a}_{ji}}{2}\zeta_j
=\frac{1+\frac{|a_{ji}|}{|a_{ij}|}}{2}\zeta_i^{-1}a_{ij}\zeta_j.
\end{equation}
It thus follows that $D_\zeta^{-1}A_HD_\zeta$ being nonnegative implies $D_\zeta^{-1}AD_\zeta$ being nonnegative, which proves the sufficiency.

{\em Necessity:} Now assume that $A$ is essentially nonnegative. That is, there exists a diagonal matrix $D_\zeta$ such that $D_\zeta^{-1}AD_\zeta$ is nonnegative. Then we have
$$
a_{ij}a_{ji}=(\zeta_i^{-1}a_{ij}\zeta_j)(\zeta_j^{-1}a_{ji}\zeta_i)\geq0
$$
from which we know that relation~\eqref{eq: relation} follows. This implies that $D_\zeta^{-1}A_HD_\zeta$ is nonnegative. This concludes the proof.
\end{proof}

The above proposition deals with the balancedness of $\G(A_H)$, instead of $\G(A)$ itself. The reason is that $\G(A)$ being balanced is not a sufficient condition for $A$ being essentially nonnegative, as shown in the following example.

\begin{exa}
{\rm Consider the complex matrix $A$ given by
$$
A=\begin{bmatrix}
    0 & 2 & 0 \\
    1 & 0 & 0 \\
    -\mathrm{j} & \mathrm{j} & 0 \\
\end{bmatrix}.
$$
It is straightforward that $\G(A)$ only has a positive cycle of length two and thus is balanced. However, we can check that $A$ is not essentially nonnegative.}
\end{exa}

The following theorem is a combination of Propositions~\ref{prop: 2} and \ref{prop: 1}.
\begin{thm}
Zero is a simple eigenvalue of $L$ with an eigenvector $\xi\in\T^n$ if and only if $\G(A)$ has a spanning tree, $\G(A_H)$ is balanced and $a_{ij}a_{ji}\geq0$ for all $1\leq i,j\leq n$.
\end{thm}

We next turn our attention to the case that $A$ is not essentially nonnegative. When $\G(A)$ has a spanning tree and $A$ is not essentially nonnegative, what we can only obtain from Proposition~\ref{prop: 2} is that either zero is not an eigenvalue of $L$, or zero is an eigenvalue of $L$ with no associated eigenvector in $\T^n$. To provide further understanding, we here consider the special case that $A$ is Hermitian. In this case, $L$ is also Hermitian.
Then all eigenvalues of $L$ are real. Let $\lambda_1\leq\lambda_2\leq\cdots\leq\lambda_n$ be the
eigenvalues of $L$. The positive semidefiniteness of $L$, i.e., the fact that $\lambda_1\geq0$, can be obtained by the following observation. For $z=[z_1,\ldots,z_n]^T\in\C^n$, we have
\begin{equation}\label{eq: positivity of L}
\begin{split}
z^*Lz&=\sum_{i=1}^n\bar{z}_i\left(\sum_{j\in\N_i}|a_{ij}|z_i-\sum_{j\in\N_i}a_{ij}z_j\right)\\
&=\frac{1}{2}\sum_{(j,i)\in\E}\left(|a_{ij}||z_i|^2+|a_{ij}||z_j|^2-2a_{ij}\bar{z}_iz_j\right)\\
&=\frac{1}{2}\sum_{(j,i)\in\E}|a_{ij}|\left|z_i-\varphi(a_{ij})z_j\right|^2\\
\end{split}
\end{equation}
where $\varphi:\ \C\backslash\{0\}\rightarrow\T$ is defined by
$\varphi(a_{ij})=\frac{a_{ij}}{|a_{ij}|}$. Based on \eqref{eq: positivity of L}, we have the following lemma.

\begin{lem}\label{lem: 2}
Let $A$ be Hermitian. Assume that $\G(A)$ has a spanning tree. Then $L$ is positive definite, i.e., $\lambda_1>0$, if and only if $A$ is not essentially nonnegative.
\end{lem}
\begin{proof}
We only show the sufficiency since the necessity follows directly from Proposition~\ref{prop: 2}. Assume the contrary. Then there exists a nonzero vector $y=[y_1,\ldots,y_n]^T\in\C^n$ such that $Ly=0$.
By \eqref{eq: positivity of L},

\begin{equation*}
\begin{split}
y^*Ly=\frac{1}{2}\sum_{(j,i)\in\E}|a_{ij}|\left|y_i-\frac{a_{ij}}{|a_{ij}|}y_j\right|^2=0.
\end{split}
\end{equation*}
This implies that $y_i=\frac{a_{ij}}{|a_{ij}|}y_j$ for $(j,i)\in\E$ and so $|y_i|=|y_j|$ for $(j,i)\in\E$. Note that for $\G(A)$ with $A$ being Hermitian, having a spanning tree is equivalent to the strong connectivity. Then we conclude that $|y_i|=|y_j|$ for all $i,j=1,\ldots,n$.
 Without loss of generality, we assume that $y\in\T^n$. It follows from Lemma~\ref{lem: 1} that $A$ is essentially nonnegative, a contradiction.
\end{proof}

On the other hand, for the general case that $A$ is not Hermitian, we cannot conclude that $L$ has no zero eigenvalue when $\G(A)$ has a spanning tree and $A$ is not essentially nonnegative. Example~\ref{exa: 1} in Section~\ref{section: examples} provides such an example.

\section{Applications}\label{section: application}
In this section, we study the distributed coordination problems with the results established in Section~\ref{section: complex Laplacians}. We first consider the consensus problems on complex-weighted digraphs.

\subsection{Complex consensus}
For a group of $n$ agents, we consider the continuous-time (CT) consensus protocol over complex field
\begin{equation}\label{eq: Algorithm}
\dot{z}_i(t)=u_i(t),\ t\geq0
\end{equation}
where $z_i(t)\in\C$ and $u_i(t)\in\C$ are the state and input of agent
$i$, respectively. We also consider the corresponding discrete-time (DT) protocol over complex field
\begin{equation}\label{eq: discrete algorithm}
z_i(k+1)=z_i(k)+u_i(k),\ k=0,1,\ldots.
\end{equation}
The communications between agents are modeled as
a complex graph $\G(A)$. The control input $u_i$  is designed, in a distributed way, as
$$
u_i=-\kappa\sum_{j\in\N_i}(|a_{ij}|z_i-a_{ij}z_j),
$$
where $\kappa>0$ is a fixed control gain. Then we have the following two systems described as
\begin{equation*}
\dot{z}_i(t)=-\kappa\sum_{j\in\N_i}(|a_{ij}|z_i-a_{ij}z_j)
\end{equation*}
and
\begin{equation*}
z_i(k+1)=z_i(k)-\kappa\sum_{j\in\N_i}(|a_{ij}|z_i-a_{ij}z_j).
\end{equation*}
Denote by $z=(z_1,\ldots,z_n)^T\in\C^n$ the aggregate position vector of $n$ agents. With the Laplacain matrix $L$ of $\G(A)$, these two systems can be rewritten in more compact forms:
\begin{equation}\label{eq: Algorithm with complex L}
\dot{z}(t)=-\kappa Lz(t)
\end{equation}
in the CT case and
\begin{equation}\label{eq: discrete algorithm with complex L}
z(k+1)=z(k)-\kappa Lz(k)
\end{equation}
in the DT case.
Inspired by the consensus in real-weighted networks \cite{Altafini13, OSM04, OSFM07}, we introduce the following definition.
\begin{defn}\label{defn: modulus consensus}
{\rm We say that the CT system \eqref{eq: Algorithm with complex L} (or the DT system~\eqref{eq: discrete algorithm with complex L}) reaches the {\em complex consensus} if $\lim_{t\rightarrow\infty}|z_i(t)|=a>0$ (or $\lim_{k\rightarrow\infty}|z_i(k)|=a>0$) for $i=1,\ldots,n$.}
\end{defn}

The following is useful in simplifying the statement of complex consensus results.
Let $A$ be an essentially nonnegative complex matrix.
If $\G(A)$ has a spanning tree, then it follows from Proposition~\ref{prop: 2} that $L$ has a simple eigenvalue at zero with an associated eigenvector $\zeta\in\T^n$. Thus, we have $A_1=D_\zeta^{-1}AD_\zeta$ is nonnegative and $D_\zeta^{-1}LD_\zeta$ has a simple eigenvalue at zero with an associated eigenvector $\mathbf{1}$.
In the standard consensus theory~\cite{RBA07}, it is well-known that $D_\zeta^{-1}LD_\zeta$ has a nonnegative left eigenvector $\nu=[\nu_1,\ldots,\nu_n]^T$ corresponding to eigenvalue zero, i.e., $\nu^T(D_\zeta^{-1}LD_\zeta)=0$ and $\nu_i\geq0$ for $i=1,\ldots,n$. We assume that $\|\nu\|_1=1$. Letting $\eta=D_\zeta^{-1}\nu=[\eta_1,\ldots,\eta_n]^T$, we have $\|\eta\|_1=1$ and $\eta^TL=0$.
We first state a necessary and sufficient condition for complex consensus of the CT system~\eqref{eq: Algorithm with complex L}.

\begin{thm}\label{thm: 2}
The CT system \eqref{eq: Algorithm with complex L} reaches complex consensus if and only if $A$ is essentially nonnegative and $\G(A)$ has a spanning tree. In this case, we have
$$
\lim_{t\rightarrow\infty}z(t)=(\eta^Tz(0))\zeta.
$$
\end{thm}

\begin{proof}
Assume that $A$ is essentially nonnegative and $\G(A)$ has a spanning tree. By Proposition~\ref{prop: 2}, we have $L$ has a simple eigenvalue at zero with an associated eigenvector $\zeta\in\T^n$. Thus, we conclude that $A_1=D_\zeta^{-1}AD_\zeta$ is nonnegative and $D_\zeta^{-1}LD_\zeta$ has a simple eigenvalue at zero with an eigenvector $\mathbf{1}$. Let $z=D_\zeta x$. By system \eqref{eq: Algorithm with complex L}, we can see that $x$ satisfies the system
\begin{equation*}
\dot{x}=-\kappa D_\zeta^{-1}LD_\zeta x.
\end{equation*}
Note that this is the standard consensus problem. From~\cite{RBA07}, it follows that
$$
\lim_{t\rightarrow\infty}x(t)=\nu^Tx(0)\mathbf{1}=\nu^TD_\zeta^{-1}z(0)\mathbf{1}.
$$
This is equivalent to
$$
\lim_{t\rightarrow\infty}z(t)=(\nu^TD_\zeta^{-1}z(0))D_\zeta\mathbf{1}=
(\eta^Tz(0))\zeta.
$$

To show the other direction,
we now assume that the system \eqref{eq: Algorithm with complex L} reaches complex consensus but $\G(A)$ does not have a spanning tree. Let $T_1$ be a maximal subtree of $\G$. Note that $T_1$ is a spanning tree of subgraph $\G_1$ of $\G(A)$. Denote by $\G_2$ the subgraph induced by vertices not belonging to $\G_1$. It is easy to see that there does not exist edge from $\G_1$ to $\G_2$ since otherwise $T_1$ is not a maximal subtree. All possible edges between $\G_1$ and $\G_2$ are from $\G_2$ to $\G_1$, and moreover we can see that there is no directed path from a vertex in $\G_2$ to the root of $T_1$ by $T_1$ being a maximal subtree again.
Therefore it is impossible to reach the complex consensus between the root of $T_1$ and vertices of $\G_2$. This implies that the system \eqref{eq: Algorithm with complex L} cannot reach complex consensus.
We obtain a contradiction. Hence $\G(A)$ have a spanning tree.
On the other hand, since the system \eqref{eq: Algorithm with complex L} reaches complex consensus we can see that the solutions $y=[y_1,\ldots,y_n]^T$ of the equation $Ly=0$ always have the property $|y_i|=|y_j|$ for all $i,j=1,\ldots,n$. Namely, zero is an eigenvalue of $L$ with an eigenvector $\zeta\in\T^n$.
It thus follows from Lemma~\ref{lem: 1} that $A$ is essentially nonnegative. We complete the proof of Theorem \ref{thm: 2}.
\end{proof}

For $\G(A)$, define the maximum modulus degree $\Delta$ by $\Delta=\max_{1\leq i\leq n}d_i$. We are now in a position to state the complex consensus result for the DT system~\eqref{eq: discrete algorithm with complex L}.

\begin{thm}\label{thm: discrete version of 2}
Assume that the input gain $\kappa$ is such that $0<\kappa<1/\Delta$.
Then the DT system \eqref{eq: discrete algorithm with complex L} reaches complex consensus if and only if $A$ is essentially nonnegative and $\G(A)$ has a spanning tree. In this case, we have
$$
\lim_{k\rightarrow\infty}z(k)=(\eta^Tz(0))\zeta.
$$
\end{thm}

\begin{proof}
Assume that $A$ is essentially nonnegative and $\G(A)$ has a spanning tree.
By Propositions~\ref{prop: 2}, we have $L$ has a simple eigenvalue at zero with an associated eigenvector $\zeta\in\T^n$. Thus, we conclude that $A_1=D_\zeta^{-1}AD_\zeta$ is nonnegative and $D_\zeta^{-1}LD_\zeta$ has a simple eigenvalue at zero with an associated eigenvector $\mathbf{1}$. Let $z=D_\zeta x$. By system \eqref{eq: discrete algorithm with complex L}, we can see that $x$ satisfies the system
\begin{equation*}
x(k+1)=(I-\kappa D_\zeta^{-1}LD_\zeta)x(k).
\end{equation*}
Note that this is the standard consensus problem. From~\cite{RBA07}, it follows that
$$
\lim_{k\rightarrow\infty}x(k)=\nu^Tx(0)\mathbf{1}=\nu^TD_\zeta^{-1}z(0)\mathbf{1}.
$$
This is equivalent to
$$
\lim_{k\rightarrow\infty}z(k)=(\nu^TD_\zeta^{-1}z(0))D_\zeta\mathbf{1}=(\eta^Tz(0))\zeta.
$$

To show the other direction,
we now assume that the system \eqref{eq: discrete algorithm with complex L} reaches
complex consensus. Using the same arguments as for the CT system \eqref{eq: Algorithm with complex L} above, we can see that $\G(A)$ have a spanning tree.
On the other hand, based on the Ger\v{s}gorin disk theorem \cite[Theorem 6.1.1]{HJ87}, all the eigenvalues
of $-\kappa L$ are located in the union of the following $n$
disks:
$$
\left\{z\in\C: \left|z+\kappa\sum_{j\in\N_i}|a_{ij}|\right|\leq\kappa\sum_{j\in\N_i}|a_{ij}|\right\}, \ i=1,\ldots,n.
$$
Clearly, all these $n$ disks are contained in the largest disk defined by
$$
\left\{z\in\C: \left|z+\kappa\Delta\right|\leq\kappa\Delta\right\}.
$$
Noting that $0<\kappa<1/\Delta$, we can see that the largest disk is contained in
the region $\left\{z\in\C: \left|z+1\right|<1\right\}\cup\{0\}$. By translation, we have all the eigenvalues of $I-\kappa L$ are located in the following region:
$$
\left\{z\in\C: |z|<1\right\}\cup\{1\}.
$$
Since the system \eqref{eq: discrete algorithm with complex L} reaches complex consensus
we can see that $1$ must be the eigenvalue of $I-\kappa L$. All other eigenvalue of $I-\kappa L$ have the modulus strictly smaller than $1$.  Moreover, if $y=[y_1,\ldots,y_n]^T$ is an eigenvector of $I-\kappa L$ corresponding to eigenvalue $1$, then $|y_i|=|y_j|>0$ for $i,j=1,\ldots,n$. That is, zero is an eigenvalue of $L$ with an eigenvalue $\zeta\in\T^n$. It thus follows from Lemma~\ref{lem: 1} that $A$ is essentially nonnegative. We complete the proof of Theorem \ref{thm: discrete version of 2}.
\end{proof}

\begin{rem}\label{rem: Hermitian}
{\rm \noindent
\begin{itemize}
\item[1)] In Theorems \ref{thm: 2} and \ref{thm: discrete version of 2}, the key point is to check the condition that $A$ is essentially nonnegative which, by Proposition~\ref{prop: 1}, can be done by examining the condition that $\G(A_H)$ is balanced and $a_{ij}a_{ji}\geq0$ for all $1\leq i,j\leq n$.

\item[2)]For the special case when $A$ is Hermitian, Theorems \ref{thm: 2} and \ref{thm: discrete version of 2} take a simpler form. As an example, we consider the CT system \eqref{eq: Algorithm with complex L} with $A$ being Hermitian. In this case, it follows from Proposition~\ref{prop: 1} that $A$ is essentially nonnegative if and only if $\G(A)$ is balanced. Then we have that the CT system \eqref{eq: Algorithm with complex L} reaches complex consensus if and only if $\G(A)$ has a spanning tree and is balanced. In this case,
$$
\lim_{t\rightarrow\infty}z(t)=\frac{1}{n}(\zeta^*z(0))\zeta.
$$
In addition, in view of Lemma~\ref{lem: 2}, it yields that  $\lim_{t\rightarrow\infty}z(t)=0$ when $\G(A)$ has a spanning tree and is unbalanced.
\item[3)] By the standard consensus results in \cite{RB05}, Theorems~\ref{thm: 2} and
\ref{thm: discrete version of 2} can be generalized to the case of switching topology. We omit the details to avoid repetitions.
\end{itemize}
}
\end{rem}

\begin{rem}
{\rm \noindent
\begin{itemize}
\item[1)]
 Theorems~\ref{thm: 2} and \ref{thm: discrete version of 2} actually give an equivalent condition to ensure that all the agents converge to a common
circle centered at the origin. Motivated by this observation, we can modify the two systems~\eqref{eq: Algorithm with complex L} and \eqref{eq: discrete algorithm with complex L} accordingly to study the circular formation problems. Similar to Theorems~\ref{thm: 2} and \ref{thm: discrete version of 2}, we can establish a necessary and sufficient condition to ensure all the agents converge to a common
circle centered at a given point and are distributed along the
circle in a desired pattern, expressed by the prespecified angle
separations and ordering among agents. We omit the details due to space limitations.
\item[2)] Part of Theorem~\ref{thm: 2} has been obtained in the literature, see~\cite[Theorems III.5 and III.6]{LH14}. As potential applications,  the reuslts in Section~\ref{section: complex Laplacians} can be used to study the set surrounding control problems~\cite{LH14}. A detailed analysis for this is beyond the scope of this paper.
\end{itemize}
}
\end{rem}

\subsection{Bipartite consensus revisited}
As an application, we now revisit some bipartite consensus results from Theorems~\ref{thm: 2} and \ref{thm: discrete version of 2}. We will see that these bipartite consensus results improve the existing results in the literature.

Let $\G(A)$ be a signed graph, i.e., $A=[a_{ij}]_{n\times n}\in\R^{n\times n}$ and $a_{ij}$ can be negative.
By bipartite consensus, we mean on a signed graph, all agents converge to a consensus value
whose absolute value is the same for all agents except for the sign.
The state $z$ is now restricted to the field
of real numbers $\R$, denoted by $x$. Then the two systems \eqref{eq: Algorithm with complex L} and \eqref{eq: discrete algorithm with complex L} reduces to the standard consensus systems:
\begin{equation}\label{eq: standard CT system}
\dot{x}(t)=-\kappa Lx(t)
\end{equation}
and
\begin{equation}\label{eq: standard DT system}
x(k+1)=x(k)-\kappa Lx(k).
\end{equation}
With the above two systems and based on Theorems~\ref{thm: 2} and \ref{thm: discrete version of 2}, we can derive the bipartite consensus results on signed graphs.
\begin{cor}\label{thm: signed 2}
Let $\G(A)$ be a signed digraph. Then the CT system \eqref{eq: standard CT system} achieves
bipartite consensus asymptotically if and only if $A$ is essentially nonnegative and
$\G(A)$ has a spanning tree. In this case, for any initial state $x(0)\in\R^n$, we have
$$
\lim_{t\rightarrow\infty}x(t)=(\eta^Tx(0))\sigma
$$
where $\sigma=[\sigma_1\ldots,\sigma_n]^T\in\{\pm1\}^n$ such that $D_{\sigma}AD_{\sigma}$ is nonnegative matrix and $\eta^TL=0$ with $\eta=[\eta_1,\dots,\eta_n]^T\in\R^n$ and $\|\eta\|_1=1$.
\end{cor}

\begin{cor}\label{thm: signed D version of 2}
Let $\G(A)$ be a signed digraph. Then the DT system \eqref{eq: standard DT system} with $0<\kappa<1/\Delta$ achieves
bipartite consensus asymptotically if and only if $A$ is essentially nonnegative and
$\G(A)$ has a spanning tree. In this case, for any initial state $x(0)\in\R^n$, we have
$$
\lim_{k\rightarrow\infty}x(k)=(\eta^Tx(0))\sigma
$$
where $\eta$ and $\sigma$ are defined as in Corollary~\ref{thm: signed 2}.
\end{cor}

\begin{rem}
{\rm Corollary \ref{thm: signed 2}  indicates that bipartite consensus can be achieved under a condition weaker than that given in Theorem~2 in~\cite{Altafini13}. In addition, we also obtain a similar necessary and sufficient condition for bipartite consensus of the DT system~\eqref{eq: standard DT system}.}
\end{rem}

\section{Examples}\label{section: examples}
In this section we present some examples to illustrate our results.
\begin{exa}
{\rm Consider the complex graph $\G(A)$ illustrated in Figure \ref{fig: balanced} with adjacency matrix
$$
A=\begin{bmatrix}
    0 & 0 & -\mathrm{j} & 0 \\
    1 & 0 & 0 & 0 \\
    0 & \mathrm{j} & 0 & 0 \\
    0 & 1+\mathrm{j} & 0 & 0 \\
\end{bmatrix}.
$$
It is trivial that $\G(A)$ has a spanning tree. Since $\G(A_H)$ is balanced, Proposition~\ref{prop: 1} implies that $A$ is essentially nonnegative. Furthermore,
defining $\zeta=[1,1,\mathrm{j},e^{\mathrm{j}\frac{\pi}{4}}]^T\in\T^4$, we have
$$
A_1=D_\zeta^{-1}AD_\zeta=\begin{bmatrix}
    0 & 0 & 1 & 0 \\
    1 & 0 & 0 & 0 \\
    0 & 1 & 0 & 0 \\
    0 & \sqrt{2} & 0 & 0 \\
\end{bmatrix}.
$$

The set of eigenvalues of the complex Laplacian $L$ is $\{0, \sqrt{2}, 3/2+\sqrt{3}\mathrm{j}/2, 3/2-\sqrt{3}\mathrm{j}/2\}$. The vector $\zeta$ is an eigenvector associated with eigenvalue zero. A simulation under system \eqref{eq: Algorithm with complex L} is given in Figure \ref{fig: Modolus consensus}, which shows that the complex consensus is reached asymptotically. This confirms the analytical results of Theorems \ref{thm: 2} and \ref{thm: discrete version of 2}.
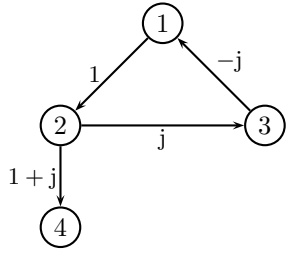
\begin{figure}[t]
\begin{center}
\begin{pspicture}(-2.5,-1)(2.5,2)
\rput(0,0){
$
\psmatrix[mnode=circle,colsep=0.8cm,rowsep=0.8cm]
  & 1 \\
2 & & 3 \\
4 \\
&
\endpsmatrix
\psset{shortput=nab,arrows=->,labelsep=0.5pt}
\small
\ncline{1,2}{2,1}
\tlput{1}
\ncline{2,3}{1,2}_{-\mathrm{j}}
\ncline{2,1}{2,3}_{\mathrm{j}}
\ncline{2,1}{3,1}
\tlput{1+\mathrm{j}}
$}
\end{pspicture}
\end{center}
\caption{Balanced graph.}
\label{fig: balanced}
\end{figure}
\begin{figure}[t]
\centering
\includegraphics[width=0.3\textwidth]{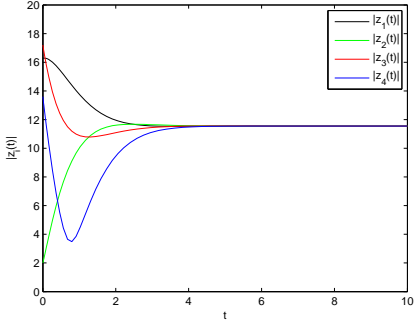}
\caption{Complex consensus process of the agents.}\label{fig: Modolus consensus}
\end{figure}}
\end{exa}

\begin{exa}\label{exa: 1}
{\rm Consider the complex graph $\G(A)$ illustrated in Figure \ref{fig: unbalanced1} with adjacency matrix
$$
A=\begin{bmatrix}
    0 & 0 & 0 & 0 & 0 & 1\\
    0 & 0 & 1 & 0 & 0 & 0 \\
    0 & 0 & 0 & 1-\mathrm{j} & 0 & 0\\
     0 & \mathrm{j} & 0 & 0 & 0 & 0\\
      \mathrm{j}& 0 & 0 & 0 & 0 & 0\\
       0 & 0 & 0 & 0 & -\mathrm{j} & 0\\
\end{bmatrix}.
$$
We can see that $\G(A)$ has a spanning tree and $A$ is not essentially nonnegative since $\G(A_H)$ is unbalanced. We can verify that zero is an eigenvalue of $L$. The simulation in Figure \ref{fig: no MC} shows that the complex consensus cannot be reached.
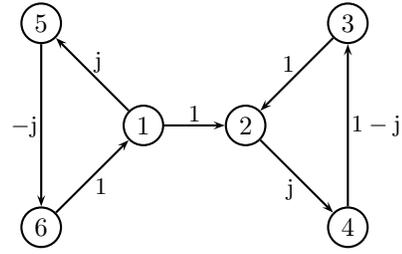
\begin{figure}[t]
\begin{center}
\begin{pspicture}(-2.5,-1)(2.5,2)
\rput(0,0){
$
\psmatrix[mnode=circle,colsep=0.8cm,rowsep=0.8cm]
  5 & & & 3\\
  & 1 & 2\\
  6 & & & 4 \\
&
\endpsmatrix
\psset{shortput=nab,arrows=->,labelsep=0.5pt}
\small
\ncline{1,1}{3,1}
\tlput{-\mathrm{j}}
\ncline{2,2}{1,1}_{\mathrm{j}}
\ncline{3,1}{2,2}_{1}
\ncline{2,2}{2,3}^{1}
\ncline{2,3}{3,4}_{\mathrm{j}}
\ncline{3,4}{1,4}_{1-\mathrm{j}}
\ncline{1,4}{2,3}_{1}
$}
 \end{pspicture}
\end{center}
\caption{Unbalanced graph.}
\label{fig: unbalanced1}
\end{figure}
\begin{figure}[t]
\centering
\includegraphics[width=0.28\textwidth]{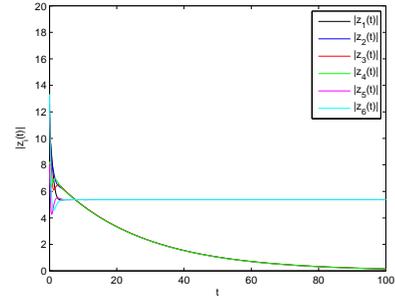}
\caption{Trajectories of the agents which mean that complex consensus cannot be reached.}\label{fig: no MC}
\end{figure}}
\end{exa}

\section{Conclusion}\label{section: conclusion}
Motivated by the study of bipartite consensus problems, we discuss the consensus problems in complex-weighted graphs. To this end, we first establish some key properties of the complex Laplacian. We emphasize that these properties can be examined by checking the properties of the corresponding digraph. Then we give some necessary and sufficient conditions to ensure the convergence of complex consensus. It is shown that these general consensus results can be used to study some distributed coordination control problems of multi-agent systems in a plane. In particular, these results cover the bipartite consensus results on signed digraphs. We believe that the properties of the complex Laplacian obtained in this paper are useful in other multi-agent coordination problems in a plane.

\end{document}